 \documentclass[final,5p, twocolumn]{elsarticle}    
 
 \usepackage{hyperref}
 \usepackage[hyperref]{xcolor}
 
 \hypersetup{
 	colorlinks   = true, 
 	urlcolor     = red, 
 	linkcolor    = blue, 
 	citecolor   = green 
 }

\usepackage{amssymb}

\usepackage{flushend}

\usepackage{mathrsfs}
\usepackage{amsmath}

\usepackage{amsthm}

\newtheorem{thm}{Theorem}
\newtheorem{lem}[thm]{Lemma}
\newtheorem{prop}[thm]{Proposition}

\newtheorem{prob}{Problem}

\newtheorem{defn}[thm]{Definition}

\theoremstyle{definition} 

\newtheorem{rem}[thm]{Remark}

\usepackage{cuted}
\usepackage{flushend}

\usepackage{graphicx}

\usepackage{enumerate}

\usepackage[hyperref]{xcolor}

\begin{document}

\begin{frontmatter}



\title{{\bf A Suboptimality Approach to Distributed $\mathcal{H}_2$ Control\\ by Dynamic Output Feedback}
}

\author{Junjie Jiao}\ead{j.jiao@rug.nl} 
\author{Harry L. Trentelman}\ead{h.l.trentelman@rug.nl} 
\author{M. Kanat Camlibel}\ead{m.k.camlibel@rug.nl}

\address{Bernoulli Institute for Mathematics, Computer Science and Artificial Intelligence, University of Groningen, 9700 AK Groningen, The Netherlands}


\begin{abstract}
This paper deals with suboptimal distributed $\mathcal{H}_2$ control by dynamic output feedback for homogeneous linear multi-agent systems.
Given a linear multi-agent system, together with an associated $\mathcal{H}_2$ cost functional, the objective is to design  dynamic output feedback protocols that guarantee the associated cost to be smaller than an a priori given upper bound while synchronizing the controlled network.
A design method is provided to compute such protocols. 
%
The computation of the two local gains in these protocols involves two Riccati inequalities, each of dimension equal to the dimension of the state space of the agents. The largest and smallest nonzero eigenvalue of the Laplacian matrix of the network graph are also used in the computation of one of the two local gains.
A simulation example is provided to illustrate the performance of the proposed  protocols.
\end{abstract}

\end{frontmatter}


\section{Introduction}\label{sec_intro}
The design of distributed protocols for networked multi-agent systems has been one of the most active research topics in the field of systems and control over the last two decades, see e.g. \cite{Olfati-Saber2004} or  \cite{6303906}. 
This is partly due to the broad range of applications of multi-agent systems, e.g. smart grids \cite{Doerfler2013},  formation control \cite{Oh2015}, \cite{YANG2019120}, and intelligent transportation systems \cite{Bart2017}. 
One of the challenging problems in the context of linear multi-agent systems is the problem of  developing distributed protocols to minimize given quadratic cost criteria while the agents reach a common goal, e.g., synchronization.
Due to the structural constraints that are imposed on the control laws by the  communication topology, such optimal control problems are difficult to solve. 
These structural constraints make distributed optimal control problems non-convex, and it is unclear under what conditions optimal solutions exist in general.

In the existing literature, many efforts have been devoted to addressing distributed linear quadratic optimal control problems.
In \cite{tamas2008},  suboptimal distributed stabilizing controllers were computed to stabilize multi-agent networks with identical agent dynamics subject to a global linear quadratic cost functional.  
For a network of agents with single integrator dynamics, an explicit expression for the optimal gain was given in \cite{wei_ren2010}, see also \cite{Jiao2019local}.
In \cite{kristian2014} and \cite{huaguang_zhang2015}, a distributed linear quadratic control problem was dealt with using an inverse optimality approach. 
This approach was further employed in \cite{nguyen_2017} to design reduced order controllers.
%
Recently, also in \cite{Jiao2018}, the suboptimal distributed LQ problem was considered.  
 In parallel to the above, much work has been put into the problem of distributed $\mathcal{H}_2$ optimal control.
Given a particular  global $\mathcal{H}_2$ cost functional,  \cite{LI2011797} and \cite{Zhongkui2014} proposed suboptimal distributed stabilizing protocols involving static state  feedback for multi-agent systems with undirected graphs. 
Later on, in  \cite{Wang2014}  these results  were generalized  to  directed graphs.
For a given $\mathcal{H}_2$ cost criterion that penalizes the weighted differences between the outputs of the communicating agents, in \cite{JIAO2018154} a suboptimal distributed synchronizing protocol based on static relative state feedback was established.  

In the past, also the design of structured controllers for large-scale systems has attracted much attention. 
In \cite{Rotkowitz2006}, the notion of quadratic invariance was adopted to develop decentralized controllers that minimize the performance of the feedback system with constraints on the controller structure.
 In \cite{Lin2013}, the so called alternating direction method of multipliers was adopted to design sparse feedback gains that minimize an $\mathcal{H}_2$ performance. 
In \cite{8344761}, conditions were provided under which, for a given optimal centralized controller, a suboptimal distributed controller exists so that the resulting closed loop state and input trajectories are close in a certain sense. 

%
%
%
The distributed $\mathcal{H}_2$  optimal control problem for multi-agent systems by dynamic output feedback is to find an optimal distributed dynamic protocol that achieves synchronization for the controlled network and that minimizes the $\mathcal{H}_2$ cost functional. 
This problem, however, is a non-convex optimization problem, and therefore it is
unclear whether such optimal protocol exists, or whether a closed form solution can be given.
Therefore, in the present paper, we look at an alternative version of this problem that requires only {\em suboptimality}.
More precisely, we extend our preliminary results from \cite{JIAO2018154} on static relative state feedback to the general case of dynamic protocols using relative measurement outputs.  
The main contributions of this paper are the following.
\begin{enumerate}[1)]
	\item We solve the open problem of finding, for a single continuous-time linear system, a separation principle based $\mathcal{H}_2$ suboptimal dynamic output feedback controller. This result extends the recent result in \cite{HAESAERT2018306} on the separation principle in suboptimal $\mathcal{H}_2$ control for discrete-time systems. 
\item  Based on the above result, we provide a method for computing $\mathcal{H}_2$ suboptimal distributed dynamic output feedback protocols for linear multi-agent systems. 
\end{enumerate}
The outline of this paper is  as follows. 
In Section \ref{sec_preliminary}, we will provide some notation and graph theory used throughout this paper. 
In Section \ref{sec_prob}, we will formulate the suboptimal distributed $\mathcal{H}_2$ control problem by dynamic output feedback for linear multi-agent systems.
%
In order to solve this problem, in Section \ref{sec_suboptimal_single_system}, we will first study suboptimal $\mathcal{H}_2$ control by dynamic output feedback for a single linear system. In Section \ref{sec_suboptimal_mas} we will then treat the problem introduced in Section \ref{sec_prob}.  %
To illustrate our method, a simulation example is provided in Section \ref{sec_simulation}.
Finally, Section \ref{sec_conclusion} concludes this paper.

\section{Preliminaries}\label{sec_preliminary}
%
\subsection{Notation}\label{subsec_notations}
In this paper, the field of real numbers is denoted by $\mathbb{R}$ and the space of $n$ dimensional real vectors is denoted by $\mathbb{R}^n$.
We denote by $\mathbf{1}_n\in \mathbb{R}^n$ the vector with all its entries equal to $1$ and we denote by $I_n$ the identity matrix of dimension $n\times n$.
For a  symmetric matrix $P$, we denote $P>0$ if $P$ is positive definite and $P < 0$  if $P$ is negative definite.
The trace of a square matrix $A$ is denoted by ${\rm tr} (A)$.
A matrix is called Hurwitz if all its eigenvalues have negative real parts. 
We denote by $\text{diag}(d_1, d_2, \ldots, d_n)$ the $n \times n$ diagonal matrix with $d_1, d_2,\ldots, d_n$ on the diagonal. For given matrices $M_1,M_2, \ldots, M_n$, we denote by ${\rm blockdiag}(M_1,M_2, \ldots, M_n)$ the block diagonal matrix with diagonal blocks $M_i$.
%
%
The Kronecker product of two matrices $A$ and $B$ is denoted by $A \otimes B$.

\subsection{Graph Theory}\label{subsec_graph}
A directed weighted graph is denoted by $\mathcal{G} = (\mathcal{V}, \mathcal{E},\mathcal{A})$ with node set $\mathcal{V} = \{ 1, 2,\ldots, N \}$ and edge set $\mathcal{E} = \{ e_1, e_2,\ldots, e_M \}$ satisfying $\mathcal{E} \subset \mathcal{V} \times \mathcal{V}$, and where $\mathcal{A} = [a_{ij}]$ is the adjacency matrix with nonnegative elements $a_{ij}$, called the edge weights. 
If $(i,j) \in \mathcal{E}$ we have $a_{ji} >0$. If  $(i,j) \not  \in \mathcal{E}$ we have $a_{ji} =0$. 

A graph is called undirected if $a_{ij}= a_{ji}$ for all $i,j$. It is called simple if $a_{ii} =0$ for all $i$.
A simple undirected graph is called connected if for each pair of nodes $i$ and $j$  there exists a path from $i$ to $j$.
%
%
%
Given a simple undirected weighted graph $\mathcal{G}$, the degree matrix of $\mathcal{G}$ is the diagonal matrix, given by $\mathcal{D} = \text{diag} ( \textsf{d}_1,\textsf{d}_2,\ldots, \textsf{d}_N )$ with $\textsf{d}_{i} = \sum_{j=1}^{N} a_{ij}$.  The Laplacian matrix is defined as ${L} := \mathcal{D} - \mathcal{A}$.
The Laplacian matrix of an undirected graph is symmetric and has only real nonnegative eigenvalues. 
A simple undirected weighted graph is connected if and only if  its Laplacian matrix ${L}$ has a simple eigenvalue at $0$. In that case there exists an orthogonal matrix $U$ such that $U^{\top} {L} U = \Lambda = \text{diag}(0, \lambda_2, \ldots, \lambda_N)$ with $0 = \lambda_1 < \lambda_2 \leq \cdots \leq \lambda_N$.
Throughout this paper, it will be a standing assumption that the communication among the agents of the network is represented by a connected, simple undirected weighted graph.

A simple undirected weighted graph obviously has an even number of edges $M$. Define $K:= \frac{1}{2}M$. For such graph, an associated incidence matrix $R \in \mathbb{R}^{N \times K}$ is defined as a matrix $R = (r_1, r_2, \ldots, r_K)$ with columns $r_k \in \mathbb{R}^{N}$. Each column $r_k$ corresponds to exactly one pair of edges $e_k = \{(i,j), (j,i)\}$, and the  $i$th and $j$th entry of $r_k$ are equal to $\pm1$, while they do not take the same value. The remaining entries of $e_k$ are equal to 0. 
%
We also define the matrix
\begin{equation}\label{W}
	W = \text{diag} ( \textsf{w}_1, \textsf{w}_2,\ldots, \textsf{w}_K)
\end{equation}
as the $K \times K$ diagonal matrix, where $\textsf{w}_k$ is the weight on each of the edges in $e_k$ for $k = 1,2,\ldots,K$.
The relation between the Laplacian matrix and the incidence matrix is captured by ${L} = R W R^{\top}$ \cite{6767074}.

\section{Problem Formulation}\label{sec_prob}
In this paper, we consider a homogeneous multi-agent system consisting of $N$ identical agents, where the underlying network  graph is a connected, simple undirected weighted graph with associated adjacency matrix $\mathcal{A}$ and  Laplacian matrix ${L}$.
The dynamics of the $i$th agent is  represented by a finite-dimensional linear time-invariant system 
\begin{equation}\label{mas_decoupled}
	\begin{aligned} 
		\dot{x}_i & = A x_i  + B u_i  + E d_i,\\
		y_i & = C_1 x_i  + D_1 d_i, \\
		z_i & = C_2 x_i  + D_2 u_i  ,
	\end{aligned}\qquad i = 1,2,\ldots,N,
\end{equation}
where $x_i \in \mathbb{R}^n$ is the state, $u_i \in \mathbb{R}^m$ is the coupling input, $d_i \in \mathbb{R}^q$ is an unknown external disturbance, $y_i \in \mathbb{R}^r$ is the measured output and $z_i \in \mathbb{R}^p$ is the output to be controlled.
The matrices $A$, $B$, $C_1$, $D_1$, $C_2$, $D_2$ and $E$ are of compatible dimensions.
%
%
Throughout this paper we assume that the pair $(A, B)$ is stabilizable and  the pair $(C_1, A)$ is detectable.
%
The agents \eqref{mas_decoupled} are to be interconnected by means of a dynamic output feedback protocol.
Following \cite{harry_2013} and \cite{Fan2016}, we consider  observer based dynamic  protocols of the form
\begin{equation}\label{dyna_protocol_1}
	\begin{aligned}
		\dot{w}_i & = A w_i + B \sum_{j=1}^N a_{ij} (u_i - u_j) \\
		&\qquad  + G \left( \sum_{j=1}^N  a_{ij}(y_i -y_j) -C_1 w_i \right),\\
		u_i & = F w_i,\quad i =1,2, \ldots, N,
	\end{aligned}
\end{equation}
where $G \in \mathbb{R}^{n \times r}$ and $F \in \mathbb{R}^{m \times n}$ are  local gains to be designed.
We briefly explain the structure of this protocol.
Each local controller of the protocol \eqref{dyna_protocol_1} observes the weighted sum of the relative input signals $\sum_{j=1}^N  a_{ij} (u_i - u_j)$ and the weighted sum of the disagreements between the measured output signals $\sum_{j=1}^N  a_{ij}(y_i -y_j)$.
The first equation in \eqref{dyna_protocol_1} in fact represents an asymptotic observer for the weighted sum of the relative states of agent $i$, and the state of this observer is an estimate of this value.
Note that, for the error $e_i := w_i - \sum_{j=1}^N  a_{ij} (x_i - x_j)$, the error dynamics is $\dot{e}_i = (A - GC_1) e_i + \sum_{j=1}^N  a_{ij} (G D_1 - E) (d_i - d_j)$.
An estimate of the weighted sum of the  relative states of each agent is then fed back to this agent using a static gain. 

Denote by
$
\textbf{x} = (x_1^{\top}, x_2^{\top}, \ldots, x_N^{\top})^{\top}
$
 the  aggregate  state vector
and likewise define $\textbf{u}$, $\textbf{y}$, $\textbf{z}$, $\textbf{d}$ and $\textbf{w}$.
The multi-agent system \eqref{mas_decoupled} can then be written in compact form as
\begin{equation}\label{compact_mas}
	\begin{aligned}  
		\dot{\textbf{x}} & = (I_N \otimes A) \textbf{x} +(I_N \otimes B) \textbf{u} + (I_N \otimes E) \textbf{d},\\
		\textbf{y} &= (I_N \otimes C_1) \textbf{x}+ (I_N \otimes D_1) \textbf{d}, \\
		\textbf{z} &= (I_N \otimes C_2)\textbf{x} + (I_N \otimes D_2) \textbf{u},
	\end{aligned}
\end{equation}
and the dynamic protocol \eqref{dyna_protocol_1} is represented by
\begin{equation}\label{compact_observer}
	\begin{aligned}
		\dot{\textbf{w}} & = \left( I_N \otimes (A - G C_1) + {L} \otimes B F \right) \textbf{w}  + ({L} \otimes G ) \textbf{y},\\
		\textbf{u} & =(I_N \otimes F) \textbf{w}.
	\end{aligned}
\end{equation}
By interconnecting the network \eqref{compact_mas} using the dynamic protocol \eqref{compact_observer}, we obtain the controlled network
	\begin{align}
		\begin{pmatrix}
			\dot{\textbf{x}} \\
			\dot{\textbf{w}}
		\end{pmatrix} 
		&=
		\begin{pmatrix}
			I_N \otimes A & I_N \otimes BF \\
			{L} \otimes GC_1 & I_N \otimes(A - G C_1) + {L} \otimes BF
		\end{pmatrix}
		\begin{pmatrix}
			\textbf{x}\\
			\textbf{w}
		\end{pmatrix} \nonumber \\
		& \qquad  +
		\begin{pmatrix}
			I_N \otimes E \\
			{L} \otimes GD_1
		\end{pmatrix} \textbf{d}
		,\label{mas_compact_1}
		\\
		\textbf{z} &= 
		\begin{pmatrix}
			I_N \otimes C_2 & I_N \otimes D_2 F	
		\end{pmatrix}
		\begin{pmatrix}
			\textbf{x}\\
			\textbf{w}
		\end{pmatrix}.\label{outputz}
	\end{align}
Foremost, we want the dynamic protocol \eqref{compact_observer} to achieve synchronization for the network.
\begin{defn}
	The protocol \eqref{compact_observer} is said to  synchronize the network if, whenever the external disturbances of all agents are equal to zero,  i.e. $\textbf{d}={0}$,  we have $x_i(t) - x_j(t) \to 0$  and $w_i(t) - w_j(t) \to 0$ as $t \to \infty$,  for all $i, j = 1, 2, \ldots, N$.
\end{defn}

The  distributed $\mathcal{H}_2$ optimal control problem by dynamic output feedback is to minimize a given global $\mathcal{H}_2$ cost functional over all dynamic protocols of the form \eqref{compact_observer} that achieve synchronization for the  controlled network. 
In the context of distributed control for multi-agent systems, we are interested in the differences of the state and output values of the agents in the controlled network, see e.g. \cite{HiddeJan2018}, \cite{6767074}.
Note that these differences are captured by the incidence matrix $R$ of the underlying graph.
Therefore, we introduce a new output variable as 
$ 
	\boldsymbol{\zeta} = (W^{\frac{1}{2}} R^{\top} \otimes I_p )\textbf{z}
$ 
with $\boldsymbol{\zeta} = (\zeta_1^{\top}, \zeta_2^{\top}, \ldots, \zeta_M^{\top})^{\top} \in \mathbb{R}^{pM}$, where $W$ is the weight matrix of the underlying graph, as defined in \eqref{W}. Thus, the output $\boldsymbol{\zeta}$ is the vector of weighted disagreements between the outputs of the agents, in which the  weights are given by the square roots of the edge weights connecting these agents.
%
Subsequently, we consider the network \eqref{mas_compact_1} with this new output:
	\begin{equation} 
		\boldsymbol{\zeta}  = 
		\begin{pmatrix}
			W^{\frac{1}{2}} R^{\top} \otimes C_2 & W^{\frac{1}{2}} R^{\top} \otimes D_2 F	
		\end{pmatrix}
		\begin{pmatrix}
			\textbf{x}\\
			\textbf{w}
		\end{pmatrix}.\label{outputzeta}
	\end{equation} 
Denote
\begin{equation*}
	\begin{aligned}
		{A}_e & =	
		\begin{pmatrix}
			I_N \otimes A & I_N \otimes BF \\
			{L} \otimes GC_1 & I_N \otimes(A - G C_1) + {L} \otimes BF
		\end{pmatrix},
		\\
		{C}_e & = 	\begin{pmatrix}
			W^{\frac{1}{2}} R^{\top} \otimes C_2 & W^{\frac{1}{2}} R^{\top} \otimes D_2 F	
		\end{pmatrix}, 
			{E}_e   = 	
	\begin{pmatrix}
	I_N \otimes E \\
	{L} \otimes GD_1
	\end{pmatrix}.
	\end{aligned}
\end{equation*}
The impulse response matrix from the external disturbance $\textbf{d}$ to the output $\boldsymbol{\zeta}$ is then equal to
\begin{equation} \label{impresponse}
T_{F,G}(t) = {C}_e e^{{A}_et}{E}_e.
\end{equation}
Next, the associated global $\mathcal{H}_2$ cost functional is defined to be the squared $\mathcal{L}_2$-norm of the closed loop impulse response, and is given by
\begin{equation}\label{cost_FG}
	J(F,G) := \int_{0}^{\infty} \text{tr}\left[ T_{F,G}^{\top}(t) T_{F,G}(t)\right] dt.
\end{equation}
The distributed $\mathcal{H}_2$ optimal  control problem by dynamic output feedback is the problem of  minimizing  \eqref{cost_FG}  over all dynamic protocols of the form \eqref{compact_observer} that achieve synchronization for the network.
Unfortunately, due to the particular form of the protocol \eqref{compact_observer}, this optimization problem is, in general, non-convex and difficult to solve, and a closed form solution has not been provided in the literature up to now. 
Therefore, instead of trying to find an optimal solution, in this paper we will address a suboptimality version of the problem. 
More specifically, we will design synchronizing dynamic protocols \eqref{compact_observer} that guarantee the associated cost \eqref{cost_FG} to be smaller than an a priori given upper bound.
More concretely, the problem that we will address is the following:
\begin{prob}\label{prob1}
	%
	Let $\gamma > 0$ be a given tolerance. 
	Design local gains $F\in \mathbb{R}^{m \times n}$ and $G\in \mathbb{R}^{n \times r}$ such that the dynamic protocol \eqref{compact_observer}  achieves $J(F, G) < \gamma$ and synchronizes the network.
\end{prob}

Before we address Problem \ref{prob1}, we will first study the suboptimal $\mathcal{H}_2$ control problem by dynamic output feedback for a single  linear system.
In that way, we will collect the required preliminary results to treat the actual suboptimal distributed $\mathcal{H}_2$ control problem for multi-agent systems.

\section{Suboptimal $\mathcal{H}_2$ Control  by Dynamic Output Feedback for Linear Systems}\label{sec_suboptimal_single_system}
In this section, we will discuss the suboptimal $\mathcal{H}_2$ control problem by dynamic output feedback for a single linear system. This problem has been dealt with before, see e.g. \cite{Scherer2000}, \cite{Scherer1997}, \cite{algebraic_1997} or \cite{HAESAERT2018306}. In particular, in \cite{HAESAERT2018306}, the {\em separation principle} for suboptimal $\mathcal{H}_2$ control for discrete-time linear systems was established. Here, we will establish the analogue of that result for  the continuous-time case. 

Consider the linear system
\begin{equation}\label{sys_xyz}
	\begin{aligned} 
		\dot{x} & = \bar{A} x + \bar{B} u + \bar{E} d,\\
		y &= \bar{C}_1 x + \bar{D}_1 d, \\
		z &= \bar{C}_2 x + \bar{D}_2 u,
	\end{aligned}
\end{equation}
where  $x \in \mathbb{R}^n$ is the state, $u \in \mathbb{R}^m$ the control input, $d\in \mathbb{R}^q$ an unknown external disturbance, $y \in \mathbb{R}^r$ the measured output, and $z \in \mathbb{R}^p$ the output to be controlled. 
The matrices $\bar{A}$, $\bar{B}$, $\bar{C}_1$, $\bar{D}_1$, $\bar{C}_2$, $\bar{D}_2$ and $\bar{E}$ have compatible dimensions. 
In this section, we assume that the pair $(\bar{A}, \bar{B})$ is stabilizable and that the pair $(\bar{C}_1, \bar{A})$ is detectable.
Moreover, we consider dynamic output feedback controllers of the form
\begin{equation}\label{dyna_w}
	\begin{aligned}
		\dot{w} &= \bar{A} w + \bar{B} u + G \left(y -\bar{C}_1 w\right), \\
		u &= F w,
	\end{aligned}
\end{equation}
where $w \in \mathbb{R}^n$ is the state of the controller, and $F \in \mathbb{R}^{m \times n}$ and $G \in \mathbb{R}^{n \times r}$ are gain matrices to be designed.
By interconnecting the controller \eqref{dyna_w} and the system \eqref{sys_xyz}, we obtain the controlled system 
\begin{equation}\label{sys_dyna_w}
	\begin{aligned}
		\begin{pmatrix}
			\dot{x}\\
			\dot{w}
		\end{pmatrix} 
		& =
		\begin{pmatrix}
			\bar{A} & \bar{B}F \\
			G \bar{C}_1 & \bar{A} +  \bar{B}F -G \bar{C}_1  
		\end{pmatrix} 
		\begin{pmatrix}
			x \\
			w
		\end{pmatrix}
		+
		\begin{pmatrix}
			\bar{E} \\
			G \bar{D}_1	
		\end{pmatrix}d ,
		\\
		z & =
		\begin{pmatrix}
			\bar{C}_2 & \bar{D}_2 F 
		\end{pmatrix}
		\begin{pmatrix}
			x \\
			w
		\end{pmatrix}.
	\end{aligned} 
\end{equation}
Denote 
\begin{align*}
	{A}_a & =  	
	\begin{pmatrix}
		\bar{A} & \bar{B}F \\
		G \bar{C}_1 & \bar{A} + \bar{B}F -G \bar{C}_1 
	\end{pmatrix} ,\\
	{C}_a  &= 
	\begin{pmatrix}
		\bar{C}_2 & \bar{D}_2 F 
	\end{pmatrix}, \\
	{E}_a  & =
	\begin{pmatrix}
		\bar{E} \\
		G \bar{D}_1	
	\end{pmatrix}.
\end{align*}
Then the impulse response matrix  from the disturbance $d$ to the output $z$ is given by ${T}_{F,G}(t) = {C}_a e^{{A}_a t}{E}_a$.
Next, we introduce the associated $\mathcal{H}_2$ cost functional, given by
\begin{equation}\label{cost_dyna}
	J(F,G)  := \int_{0}^{\infty} 
	\text{tr} \left[ {T}_{F,G}^{\top}(t) {T}_{F,G}(t) \right] dt.
\end{equation}
We are interested in the problem of finding a controller of the form \eqref{dyna_w} such that  the controlled system \eqref{sys_dyna_w} is internally stable and  the associated cost \eqref{cost_dyna} is smaller than an a priori given upper bound. 

%

Before we proceed, we first review a well-known result that provides {\em necessary and sufficient} conditions such that a closed loop system is  $\mathcal{H}_2$ suboptimal, see e.g. \cite[Proposition 3.13]{Scherer2000}.
\begin{prop}\label{prop_Ae} 
	%
	Let $\gamma >0$.
	%
	%
	Then the following statements are equivalent:
	\begin{enumerate}[(i)]
		\item the system \eqref{sys_dyna_w} is internally stable and  $J(F, G)  <\gamma$.
		\item \label{case2} there exists $X_a >0$ such that 
		\begin{equation}\label{prop_ineq_1}
			\begin{aligned}
				{A}_a X_a + X_a {A}_a^\top + {E}_a {E}_a^\top  &<0,\\
				\text{\em tr}\left({C}_aX_a{C}_a^\top\right)  &< \gamma.
			\end{aligned}
		\end{equation}
		\item \label{case3} there exists $Y_a >0$ such that
		\begin{equation}
			\begin{aligned}
				{A}_a^\top Y_a + Y_a {A}_a + {C}_a^\top {C}_a &<0, \\
				\text{\em tr}\left({E}_a^\top Y_a {E}_a\right) &< \gamma.
			\end{aligned}
		\end{equation} 
	\end{enumerate}	
\end{prop} 
%
%

%
The following lemma is an extension of Theorem 6 in \cite{HAESAERT2018306}. It provides conditions under which the  controller \eqref{dyna_w} with gain matrices $F$ and $G = Q\bar{C}_1^\top$  is suboptimal for the continuous-time system \eqref{sys_xyz}, where $Q$ is a particular real symmetric solution of a given Riccati inequality. The result shows that the separation principle is also applicable in the context of suboptimal $\mathcal{H}_2$ control for continuous-time systems.

%

\begin{lem}\label{lem_F_G}
	%
	Let $\gamma > 0$ be a given tolerance. 
	Assume that $\bar{D}_1\bar{E}^{\top}  =0$, $\bar{D}_2^{\top} \bar{C}_2 =0$, $\bar{D}_1  \bar{D}_1^{\top} =I_r$ and $\bar{D}_2^\top \bar{D}_2 >0$. 
	Let  $F\in \mathbb{R}^{m \times n}$. Suppose that there exists $P > 0$ satisfying
	\begin{equation}\label{are_P}
		(\bar{A} + \bar{B} F)^\top P + P (\bar{A} + \bar{B} F) + (\bar{C}_2 + \bar{D}_2 F)^\top  (\bar{C}_2 + \bar{D}_2 F)   < 0. 
	\end{equation}
	Let $Q > 0$ be a solution of the Riccati inequality
	\begin{equation}\label{are_Q}
		\bar{A} Q + Q \bar{A}^\top - Q \bar{C}_1^\top \bar{C}_1Q + \bar{E} \bar{E}^\top < 0.
	\end{equation}
	If, moreover, the inequality 
	\begin{equation}\label{gamma_P_Q}	
		\text{\em tr} \left( \bar{C}_1Q   P Q \bar{C}_1^\top  \right) + \text{\em tr} \left( \bar{C}_2 Q  \bar{C}_2^\top  \right) < \gamma
	\end{equation}
	holds, then the controller \eqref{dyna_w} with the gains $F$ and  $G = Q \bar{C}_1^\top$ yields an internally stable closed loop system \eqref{sys_dyna_w}, and it  is suboptimal, i.e. $J(F, G) <\gamma$.
\end{lem}


\begin{proof}
	Let  $Q > 0$ satisfy \eqref{are_Q} and gain matrix  $F$ be given.
	Note that \eqref{gamma_P_Q} is equivalent to	\begin{equation}\label{ineq_gamma}
		\text{tr} \left( \bar{C}_1Q   P Q \bar{C}_1^\top  \right) < \gamma - \text{tr} \left( \bar{C}_2 Q  \bar{C}_2^\top  \right).
	\end{equation}
	According to cases \eqref{case2} and \eqref{case3} in Proposition \ref{prop_Ae}, there exists  $P>0$ satisfying \eqref{are_P} and \eqref{ineq_gamma}
	if and only if there exists $\Delta >0$ satisfying
	\begin{equation}\label{cost}
		\text{tr}\left((\bar{C}_2 + \bar{D}_2 F) \Delta  (\bar{C}_2+ \bar{D}_2 F)^\top\right) < \gamma - \text{tr} \left( \bar{C}_2 Q  \bar{C}_2^\top  \right)
	\end{equation}
	and
	\begin{equation}\label{F_Q}
		(\bar{A} + \bar{B} F)  \Delta + \Delta (\bar{A} + \bar{B} F)^\top +Q \bar{C}_1^\top \bar{C}_1 Q  < 0. 
	\end{equation}
	On the other hand, by applying the state transformation
	\begin{equation*}
		\begin{pmatrix}
			w \\
			e
		\end{pmatrix}
		=
		\begin{pmatrix}
			0 & I_n \\
			-I_n & I_n
		\end{pmatrix}
		\begin{pmatrix}
			x\\
			w
		\end{pmatrix}.
	\end{equation*}
	The system \eqref{sys_dyna_w} then becomes
	\begin{equation}\label{sys_dyna_new}
		\begin{aligned}
			\begin{pmatrix}
				\dot{w}\\
				\dot{e}
			\end{pmatrix} 
			& =
			\begin{pmatrix}
				\bar{A} +  \bar{B}F  & - G \bar{C}_1  \\
				0 & \bar{A} -G \bar{C}_1 
			\end{pmatrix} 
			\begin{pmatrix}
				w \\
				e
			\end{pmatrix}
			+
			\begin{pmatrix}
				G \bar{D}_1	\\
				G \bar{D}_1 - \bar{E} 
			\end{pmatrix}d ,
			\\
			z & =
			\begin{pmatrix}
				\bar{C}_2 + \bar{D}_2 F & -\bar{C}_2
			\end{pmatrix}
			\begin{pmatrix}
				w \\
				e
			\end{pmatrix}.
		\end{aligned} 
	\end{equation}
	Clearly, the system \eqref{sys_dyna_w} is internally stable  if and only if $\bar{A} +  \bar{B}F $ and $\bar{A} -G \bar{C}_1 $ are Hurwitz.
	Thus, what remains to show   is that the controller \eqref{dyna_w} with  the gains $F$ and  $G = Q \bar{C}_1^\top$ internally stabilizes the system \eqref{sys_xyz} and that $J(F, G) <\gamma$.

	Note that \eqref{are_Q} is equivalent to 
	\begin{equation}\label{are_Q1}
		\begin{aligned}
			(\bar{A}-Q \bar{C}_1^\top \bar{C}_1) Q + Q( \bar{A}-Q \bar{C}_1^\top \bar{C}_1)^\top  &\\
			+ (\bar{E} + Q \bar{C}_1^\top \bar{D}_1)  (\bar{E} + Q \bar{C}_1^\top \bar{D}_1)^\top  &< 0,
		\end{aligned}
	\end{equation}
	where we use the fact that $\bar{D}_1\bar{E}^{\top}  =0$ and  $\bar{D}_1  \bar{D}_1^{\top} =I_r$.
	Since $G = Q \bar{C}_1^\top$, it then  follows that $\bar{A} -G \bar{C}_1 $ is Hurwitz.
	Similarly, it follows from \eqref{are_P} that $\bar{A} +  \bar{B}F $ is Hurwitz.
	Consequently, the system  \eqref{sys_dyna_w} is internally stable.
	
	Next, we will show that  $J(F,G) <\gamma$. Again consider \eqref{sys_dyna_new} and denote 
	\begin{align*}
		\bar{A}_a & =  	
		\begin{pmatrix}
			\bar{A} +  \bar{B}F  & - G \bar{C}_1  \\
			0 & \bar{A} -G \bar{C}_1 
		\end{pmatrix} ,\\
		\bar{C}_a  &= 
		\begin{pmatrix}
			\bar{C}_2 + \bar{D}_2 F & -\bar{C}_2
		\end{pmatrix}, \\
		\bar{E}_a  &=
		\begin{pmatrix}
			G \bar{D}_1	\\
			G \bar{D}_1 - \bar{E} 
		\end{pmatrix}.
	\end{align*}
	According to Proposition \ref{prop_Ae}, in particular the inequalities in \eqref{prop_ineq_1}, we have $J(F,G) <\gamma$ if and only if there exists $P_a>0$ satisfying
	\begin{equation}\label{ineq_Pe}		
		\begin{aligned}
			\bar{A}_a P_a + P_a \bar{A}_a^\top + \bar{E}_a \bar{E}_a^\top &<0, \\
			{\rm tr}(\bar{C}_a P_a \bar{C}_a^\top) & < \gamma.
		\end{aligned}
	\end{equation}
	We  will show that the existence of solutions $Q>0$ and $\Delta >0$ to the inequalities \eqref{are_Q}, \eqref{cost} and \eqref{F_Q} implies that  \eqref{ineq_Pe} has a solution $P_a>0$.
	Let 
	\[
	P_a = 
	\begin{pmatrix}
	\Delta & 0 \\
	0 & Q
	\end{pmatrix}.
	\]
	Clearly, $P_a >0$.
	By substituting $P_a$ , $\bar{A}_a$, $\bar{E}_a$ and $\bar{C}_a$ into \eqref{ineq_Pe}, we obtain 
	\begin{equation}\label{R}
		\begin{pmatrix}
			R_1 & R_{12} \\
			R_{12}^\top & R_2
		\end{pmatrix} <0,
	\end{equation}
	where
	\begin{align*}
		&R_1 = (\bar{A} + \bar{B}F) \Delta + \Delta (\bar{A} -\bar{B}F)^\top + GG^\top, \\
		&R_{12} =- G \bar{C}_1 Q +  GG^\top , \\
		&R_2 = (\bar{A}-G\bar{C}_1) Q  + Q (\bar{A}-G\bar{C}_1)^\top \\
		&\qquad\quad + (G\bar{D}_1 - \bar{E})(G\bar{D}_1 - \bar{E})^\top,
	\end{align*}
	and 
	\begin{equation}\label{cost_F_C}
		{\rm tr}\left((\bar{C}_2 + \bar{D}_2 F) \Delta  (\bar{C}_2+ \bar{D}_2 F)^\top\right) + {\rm tr}\left( \bar{C}_2 Q \bar{C}_2^\top\right) <\gamma. 
	\end{equation}
	%
	It then follows from $G = Q \bar{C}_1^\top$,  \eqref{are_Q} and \eqref{F_Q} that $R_1 <0$, $R_{12} =0$ and $R_2 <0$. Subsequently,  $R <0$.
	Also, it follows from \eqref{cost} that \eqref{cost_F_C} holds.
	Hence,  $J(F,G) <\gamma$.
	This completes the proof.
\end{proof}
\begin{thm}\label{linear_h2}
	%
	Let $\gamma > 0$.
	%
	%
	Assume that $\bar{D}_1 \bar{E}^{\top} =0$, $\bar{D}_2^{\top} \bar{C}_2 =0$ and  $\bar{D}_1  \bar{D}_1^{\top} =I_r$,  $\bar{D}_2^\top \bar{D}_2>0$. 
	Suppose that there exist $P >0$ and $Q>0$ satisfying
	\begin{align}
		\bar{A}^\top P + P \bar{A} -P \bar{B} (\bar{D}_2^\top \bar{D}_2)^{-1}  \bar{B}^\top P + \bar{C}_2^\top \bar{C}_2 & <0,  \label{P2}\\
		\bar{A} Q + Q \bar{A}^\top - Q \bar{C}_1^\top  \bar{C}_1 Q + \bar{E} \bar{E}^\top  & <0, \label{Q2}\\
		\text{\em tr} \left(\bar{C}_1 Q  P Q \bar{C}_1^\top \right) + \text{\em tr} \left( \bar{C}_2 Q  \bar{C}_2^\top  \right) &< \gamma. \label{gamma_2}
	\end{align}
	Let $G = Q \bar{C}_1^\top$ and $F = -(\bar{D}_2^\top \bar{D}_2)^{-1}\bar{B}^\top P$. 
	Then the controller \eqref{dyna_w} internally stabilizes the system \eqref{sys_xyz}, and it is suboptimal, i.e. $J(F, G) <\gamma$.
\end{thm}
\begin{proof}
	Substituting  $F = -(\bar{D}_2^\top \bar{D}_2)^{-1}\bar{B}^\top P$ into \eqref{are_P} gives us the inequality \eqref{P2}.
	The rest follows from Lemma \ref{lem_F_G}.
\end{proof}


We are now ready to deal with the suboptimal distributed  $\mathcal{H}_2$ control problem by dynamic output feedback for multi-agent systems.

\section{Suboptimal Distributed $\mathcal{H}_2$ Control for Multi-Agent Systems by Dynamic Output Feedback} \label{sec_suboptimal_mas}
In this section, we will address Problem \ref{prob1}.
%
For the multi-agent system \eqref{mas_decoupled}, we will establish a design method for local gains $F$ and $G$ such that the protocol \eqref{dyna_protocol_1} achieves $J(F,G) < \gamma$ and synchronizes the network \eqref{mas_compact_1}. 

Let $U$ be an orthogonal matrix such that $U^{\top} {L} U = \Lambda = \text{diag}(0, \lambda_2, \ldots, \lambda_N)$ with $0 = \lambda_1 < \lambda_2 \leq \cdots \leq \lambda_N$ the eigenvalues of the Laplacian matrix. We apply the state transformation
\begin{equation}\label{state_trans1}
	\begin{pmatrix}
		\bar{\textbf{x}} \\
		\bar{\textbf{w}}
	\end{pmatrix} 
	=
	\begin{pmatrix}
		U^{\top} \otimes I_n & 0 \\
		0 & U^{\top} \otimes I_n	
	\end{pmatrix}
	\begin{pmatrix}
		\textbf{x} \\
		\textbf{w}
	\end{pmatrix}.
\end{equation}
Then the controlled network \eqref{mas_compact_1} with the associated output \eqref{outputzeta} is also represented by 
	\begin{align} 
		\begin{pmatrix}
			\dot{\bar{\textbf{x}}} \\
			\dot{\bar{\textbf{w}}}
		\end{pmatrix} 
		&=
		\begin{pmatrix}
			I_N \otimes A & I_N \otimes BF \\
			\Lambda \otimes GC_1 & I_N \otimes(A - G C_1) + \Lambda \otimes BF
		\end{pmatrix}
		\begin{pmatrix}
			\bar{\textbf{x}}\\
			\bar{\textbf{w}}
		\end{pmatrix} \nonumber\\
		&\qquad +
		\begin{pmatrix}
			U^{\top} \otimes E \\
			U^{\top}L \otimes GD_1
		\end{pmatrix} \textbf{d}
		,\nonumber
		\\
		\boldsymbol{\zeta} &= 
		\begin{pmatrix}
			W^{\frac{1}{2}} R^{\top}  U \otimes C_2 & W^{\frac{1}{2}} R^{\top}  U  \otimes D_2 F	
		\end{pmatrix}
		\begin{pmatrix}
			\bar{\textbf{x}}\\
			\bar{\textbf{w}}
		\end{pmatrix}.\label{mas_compact_trans}
	\end{align}
%
Denote
\begin{equation*}
	\begin{aligned}
		\bar{A}_e&  = 
		\begin{pmatrix}
			I_N \otimes A & I_N \otimes BF \\
			\Lambda \otimes GC_1 & I_N \otimes(A - G C_1) + \Lambda \otimes BF
		\end{pmatrix},
		\\
		\bar{C}_e & =
		\begin{pmatrix}
			W^{\frac{1}{2}} R^{\top}  U \otimes C_2 & W^{\frac{1}{2}} R^{\top}  U  \otimes D_2 F	
		\end{pmatrix},\\
				\bar{E}_e & =
		\begin{pmatrix}
		U^{\top} \otimes E \\
		U^{\top}L \otimes GD_1
		\end{pmatrix}.
	\end{aligned}
\end{equation*}
Obviously, the impulse response matrix $T_{F,G}(t)$ given by \eqref{impresponse} is then equal to 
$\bar{C}_e e^{\bar{A}_e t} \bar{E}_e$.

In order to proceed, we now introduce the $N-1$ auxiliary linear systems
\begin{equation}\label{new_systems}
	\begin{aligned}
		\dot{\xi}_i & = A \xi_i + \lambda_i B v_i + E \delta_i ,\\
		\vartheta_i & =  C_1 \xi_i + D_1 \delta_i, \\
		\eta_i & = \sqrt{\lambda_i} C_2 \xi_i  + \lambda_i \sqrt{\lambda_i} D_2 v_i,
	\end{aligned}
\end{equation}
and associated dynamic output feedback controllers
\begin{equation}\label{protocol_new_sys}
	\begin{aligned}
		\dot{\omega}_i & = A \omega_i  + \lambda_i B v_i + G (\vartheta_i- C_1 \omega_i) ,\\
		v_i & = F \omega_i, \quad i = 2, 3, \ldots, N
	\end{aligned}
\end{equation}
with gain matrices $F$ and $G$. 
	By interconnecting \eqref{protocol_new_sys} and \eqref{new_systems}, we obtain the $N-1$ closed loop systems 
\begin{align}
\begin{pmatrix}
\dot{\xi_i }\\
\dot{\omega}_i
\end{pmatrix}
& = 
\begin{pmatrix}
A & \lambda_i B F \\
GC_1 & A -GC_1 + \lambda_i B F 
\end{pmatrix}
\begin{pmatrix}
\xi_i \\
\omega_i
\end{pmatrix}
+
\begin{pmatrix}
E \\
G D_1
\end{pmatrix} \delta_i , \nonumber
\\
\eta_i &= 
\begin{pmatrix}
\sqrt{\lambda_i} C_2 & \lambda_i \sqrt{\lambda_i} D_2 F
\end{pmatrix}
\begin{pmatrix}
\xi_i \\
\omega_i
\end{pmatrix},\label{sys_new2}
\end{align}
for $ i = 2, 3,\ldots,N$.
The impulse response matrix of \eqref{sys_new2} from the disturbance $\delta_i$ to the output $\eta_i$ is equal to 
\begin{equation} \label{help}
T_{i, F,G}(t) = \bar{C}_i e^{\bar{A}_i t} \bar{E}_i
\end{equation}
with
$\bar{A}_i   = 	
\begin{pmatrix}
A & \lambda_i B F \\
GC_1 & A -GC_1 + \lambda_i B F 
\end{pmatrix}, $
$
\bar{E}_i   = 	
\begin{pmatrix}
E \\
G D_1
\end{pmatrix},$
$
\bar{C}_i  = 
\begin{pmatrix}
\sqrt{\lambda_i} C_2 & \lambda_i\sqrt{\lambda_i} D_2 F
\end{pmatrix}$.
Furthermore, for each system \eqref{new_systems} the associated $\mathcal{H}_2$ cost functional is given by
\begin{equation}\label{cost_new_sys_2}
\begin{aligned}
{J}_i(F,G) : = \int_{0}^{\infty} \text{tr} \left[T_{i,F,G}^{\top}(t) T_{i,F,G}(t) \right] dt, \
i= 2,3,\ldots,N. 
\end{aligned}
\end{equation}
Then we have the following lemma:
\begin{lem}\label{lem_cost_2}
		Let $F \in \mathbb{R}^{m \times n}$ and $G \in \mathbb{R}^{n \times r}$. Then the dynamic protocol \eqref{dyna_protocol_1} with gain matrices $F$ and $G$ achieves synchronization for the network \eqref{mas_compact_1} if and only if for each $i = 2, 3, \ldots, N$ the controller \eqref{protocol_new_sys} with gain matrices $F$ and $G$  internally stabilizes the system \eqref{new_systems}.  
	Moreover, we have 
	\begin{equation}\label{cost_costs2}
		J(F,G) = \sum_{i=2}^{N} {J}_i (F,G).
	\end{equation}
\end{lem}
\begin{proof}
	It follows immediately from \cite[Lemmas 3.2 and 3.3]{harry_2013} that the dynamic protocol \eqref{dyna_protocol_1} achieves synchronization for the network \eqref{mas_compact_1} if and only if  for  $i = 2, 3, \ldots, N$ the system  \eqref{new_systems} is internally stabilized by the controller  \eqref{protocol_new_sys}. 
	
	Next, we prove \eqref{cost_costs2}.
Let $F$ and $G$ be such that synchronization is achieved. Then we have 
\[
J(F,G) = \int_0^\infty \text{tr} \left( {\bar{E}_e^\top e^{\bar{A}^\top_e t} \bar{C}_e^\top \bar{C}_e e^{\bar{A}_e t} \bar{E}_e}\right)  dt.
\]
Since $U^{\top} {L} U = \Lambda$, ${L} = R W R^{\top}$, we have	
$\bar{C}_e^\top \bar{C}_e =  \tilde{C}_e^\top \tilde{C}_e$
with 
$
\tilde{C}_e := \left( \Lambda^{\frac{1}{2}} \otimes C_2 ~~\Lambda^{\frac{1}{2}} \otimes D_2 F \right)
$.	
We also have $\bar{E}_e  \bar{E}^\top_e = \tilde{E}_e  \tilde{E}^\top_e$ with 
$
\tilde{E}_e := 
\begin{pmatrix}
I_N \otimes E \\
\Lambda \otimes G D_1
\end{pmatrix}
$. 
Thus we find that
\begin{equation} \label{step}
\textrm{tr}\left( {\bar{E}_e^\top e^{\bar{A}^\top_e t} \bar{C}_e^\top \bar{C}_e e^{\bar{A}_e t} \bar{E}_e}\right) =
\textrm{tr}\left( {\tilde{E}_e^\top e^{\bar{A}^\top_e t} \tilde{C}_e^\top \tilde{C}_e e^{\bar{A}_e t} \tilde{E}_e}\right).
\end{equation}
We now analyze the matrix function $\tilde{C}_e e^{\bar{A}_e t} \tilde{E}_e$ appearing in \eqref{step}.
By applying suitable permutations of the blocks appearing in the matrices $\tilde{C}_e$, $\tilde{E}_e$ and $\bar{A}_e$, it is straightforward to show that 
$
\tilde{C}_e e^{\bar{A}_e t} \tilde{E}_e = {\rm blockdiag} \left(0, C_2 e^{A_2 t} E_2, \ldots, C_N e^{A_N t} E_N \right),
$
where 
\begin{equation*}
\begin{aligned}
{A}_i & := 	
\begin{pmatrix}
A &  B F \\
\lambda_iGC_1 & A -GC_1 + \lambda_i B F 
\end{pmatrix}, 
\\
{C}_i & := 
\begin{pmatrix}
\sqrt{\lambda_i} C_2 & \sqrt{\lambda_i} D_2 F
\end{pmatrix},~~
{E}_i  := 	
\begin{pmatrix}
E \\
\lambda_i G D_1
\end{pmatrix}.
\end{aligned}
\end{equation*}
It is easily seen that for $i = 2,3, \ldots, N$ the systems $(A_i,E_i,C_i)$ and  $(\bar{A}_i,\bar{E}_i,\bar{C}_i)$ are isomorphic. Hence they have the same impulse response $T_{i,F,G}(t)$, which is given by \eqref{help}, see e.g.,  \cite[Theorem 3.10]{harry_book}.
As a consequence we obtain that 
$
\tilde{C}_e e^{\bar{A}_e t} \tilde{E}_e = {\rm blockdiag} \left(0, T_{2,F,G}(t), \ldots, T_{N,F,G}(t) \right).
$
Thus we find that
\[
J(F,G) = \int_0^{\infty} \sum_{i =2}^{N}  \text{tr} \left[T_{i,F,G}^{\top}(t) T_{i,F,G}(t) \right] dt.
\]
The claim \eqref{cost_costs2} then follows immediately. 
%
\end{proof}
By applying Lemma \ref{lem_cost_2},  we have transformed the suboptimal distributed $\mathcal{H}_2$ control problem by dynamic output feedback for the multi-agent network \eqref{mas_compact_1} into  suboptimal $\mathcal{H}_2$ control problems for the $N-1$ linear systems \eqref{new_systems} using controllers \eqref{protocol_new_sys} with the same gain matrices $F$ and $G$.
%
Next, we establish conditions under which the $N-1$ systems \eqref{new_systems} are internally stabilized by their corresponding controllers  \eqref{protocol_new_sys}  for $i = 2, 3, \ldots, N$, while achieving $\sum_{i=2}^{N} {J}_i (F, G) < \gamma$.

\begin{lem}\label{lem_FFG}
Let $\gamma > 0$ be a given tolerance. 
	Assume that ${D}_1 {E}^{\top} =0$, ${D}_2^{\top} {C}_2 =0$, ${D}_1  {D}_1^{\top} =I_r$ and ${D}_2^\top {D}_2 =I_m$. 
		 For  $i =2,3, \ldots, N$, let  $F$,  $P_i>0$, and $Q>0$ be such that  the inequalities 
	\begin{align}
		(A + \lambda_i B F )^\top P_i  + P_i (A + \lambda_i B F )  \qquad\qquad\qquad\quad&\nonumber\\
		 + (\sqrt{\lambda_i}C_2+ \lambda_i \sqrt{\lambda_i} D_2 F )^\top (\sqrt{\lambda_i}{C}_2 + \lambda_i \sqrt{\lambda_i} {D}_2 F )  &< 0, \label{P}\\
		A Q + QA^\top - Q C_1^\top C_1Q + E E^\top < 0, &  \label{Q}\\
		\sum_{i=2}^N 
		\left[
		\text{\em tr} \left(C_1 Q P_i QC_1^\top  \right) + \lambda_i \text{\em tr} \left(  C_2 Q  C_2^\top  \right) \right]  < \gamma &\label{gamma}
	\end{align}
	hold.
	Then for each $i=2,3, \ldots, N$, the controller \eqref{protocol_new_sys} with  gain matrices $F$ and $G = Q C_1^\top$ internally stabilizes the system  \eqref{new_systems}, and, moreover, $\sum_{i=2}^{N} {J}_i (F, G) < \gamma$.
\end{lem}
\begin{proof}
	By \eqref{gamma}, for  $\epsilon_i >0$  sufficiently small, we have $\sum_{i=2}^N \gamma_i <\gamma$, where $\gamma_i := \text{tr} \left(C_1 Q P_i QC_1^\top  \right) + \lambda_i \text{tr} \left(  C_2 Q  C_2^\top  \right) + \epsilon_i$.
Since 
\[
\text{tr} \left(C_1 Q P_i QC_1^\top  \right) + \lambda_i  \text{tr} \left( C_2 Q  C_2^\top  \right) < \gamma_i,
\]
by taking $\bar{A} = A$, $\bar{B} = \lambda_i B$, $\bar{C}_1 = C_1$, $\bar{D}_1 = D_1$, $\bar{C}_2 = \sqrt{\lambda_i} C_2$, $\bar{D}_2 = \lambda_i \sqrt{\lambda_i} D_2$, $\bar{C}_1 = C_1$ and $\bar{E} =E$ in Lemma \ref{lem_F_G}, it follows that the controller \eqref{protocol_new_sys} internally stabilizes   the system  \eqref{new_systems} and $ {J}_i (F, G) < \gamma_i$.
Thus, from $\sum_{i=2}^N \gamma_i <\gamma$  it follows  that $\sum_{i=2}^{N} {J}_i (F, G) < \gamma$.
%
\end{proof}

Again, we note that the four conditions $D_1E^{\top}  =0$, $D_2^{\top} C_2 =0$, $D_1 D_1^{\top} =I_r$ and ${D}_2^\top {D}_2 =I_m$ are made here to simplify notation, and can be replaced by the regularity conditions $D_1 D_1^{\top} >0$ and ${D}_2^\top {D}_2 >0$ alone.

By combining Lemma \ref{lem_cost_2} and Lemma \ref{lem_FFG}  we have established sufficient conditions for given gain matrices $F$ and $G$ to synchronize the network  \eqref{mas_compact_1} and to be suboptimal, i.e. $J(F, G) < \gamma$. In fact, $G$ is taken to be equal to  $Q C_1^\top$, with $Q>0$ a solution to the Riccati inequality \eqref{Q}.
However, no design method has yet been provided  to compute a suitable matrix $F$.
In the following theorem, we will establish a design method for computing such gain matrix $F$. Together with $G$ given above, this will lead to a distributed suboptimal protocol for multi-agent system (\ref{mas_decoupled}) with associated cost functional (\ref{cost_FG}). 

\begin{thm}\label{main1}
Let $\gamma > 0$ be a given tolerance. 
	Assume that ${D}_1 {E}^{\top} =0$, ${D}_2^{\top} {C}_2 =0$, ${D}_1  {D}_1^{\top} =I_r$ and ${D}_2^\top {D}_2 =I_m$. 
	%
	%
	Let $Q > 0$ satisfy
	\begin{equation}\label{one_are_q}
	A Q + Q A^{\top} - Q C_1^{\top} C_1 Q + E^{\top} E < 0.
	\end{equation}
	Let $c$ be any real number such that $0 < c <\frac{2}{\lambda_N^2}$.
	We distinguish two cases:
	\begin{enumerate}[(i)]
		\item \label{thm1_case1}
		if 
		\begin{equation}\label{c1}
			\frac{2}{\lambda_2^2 + \lambda_2 \lambda_N + \lambda_N^2} \leq c <\frac{2}{\lambda_N^2}
		\end{equation}
		then there exists $P>0$ satisfying 
		\begin{equation}\label{one_are_p}
		A^{\top} P + P A + (c^2 \lambda_N^3 - 2 c \lambda_N) PB B^{\top} P +\lambda_N C_2^{\top} C_2 <0.
		\end{equation}
		\item \label{thm1_case2}
		if 	
		\begin{equation}\label{c2}
			0 < c < \frac{2}{\lambda_2^2 + \lambda_2 \lambda_N + \lambda_N^2}
		\end{equation}
		then there exists $P>0$ satisfying 
		\begin{equation}\label{are_p2}
			A^{\top} P + P A + (c^2 \lambda_2^3 - 2 c \lambda_2) PB B^{\top} P +\lambda_N C_2^{\top} C_2 <0.
		\end{equation}
	\end{enumerate}
	In both cases, if in addition $P$ and $Q$ satisfy
	\begin{equation}\label{one_gamma}
		\text{\em tr} \left(C_1 Q  P Q C_1^\top \right) +  \lambda_N \text{\em tr} \left(  C_2 Q  C_2^\top  \right)  < \frac{\gamma}{N-1}.
	\end{equation}
	Then the protocol \eqref{dyna_protocol_1} with $F := - c B^{\top} P$ and $G := Q C_1^{\top}$  synchronizes the network \eqref{mas_compact_1} and it is suboptimal, i.e. $J(F, G) < \gamma$.
\end{thm}
\begin{proof}
We will only provide the proof for case \eqref{thm1_case1} above.
	Using the upper and lower bound on $c$ given by (\ref{c1}), it can be verified that $c^2 \lambda_N^3 - 2 c \lambda_N <0$. 
	Thus the Riccati inequality \eqref{one_are_p} has positive definite solutions.
	Since $c^2 \lambda_i^3 - 2 c \lambda_i \leq c^2 \lambda_N^3 - 2 c \lambda_N <0$ and $\lambda_i \leq \lambda_N$ for  $i=2,3,\ldots,N$, any positive definite solution $P$ of (\ref{one_are_p}) also satisfies the $N-1$ Riccati inequalities
	\begin{equation}\label{n_are1}
		A^{\top} P + P A + (c^2 \lambda_i^3 - 2 c \lambda_i) PB B^{\top} P +\lambda_i C_2^{\top} C_2 <0, 
	\end{equation}
	equivalently, 
	\begin{equation}\label{one_lyap_p1}
		\begin{aligned}
			(A-c\lambda_i B B^{\top}P)^{\top}P  + P(A- c\lambda_iB B^{\top}P) &\\
			  + c^2\lambda_i^3 PB B^{\top}P + \lambda_i C_2^{\top} C_2 &< 0,
		\end{aligned}
	\end{equation}
	for $i=2,\ldots,N$.
	Using the conditions ${D}_2^{\top} {C}_2 =0$ and ${D}_2^\top {D}_2 =I_m$ this yields
	\begin{equation}\label{one_lyap_p}
		\begin{aligned}
	&	(A-c\lambda_i B B^{\top}P)^{\top}P  + P(A- c\lambda_iB B^{\top}P) \\
	&\qquad	+(\sqrt{\lambda_i}{C}_2 + \lambda_i \sqrt{\lambda_i} {D}_2 B^\top P)^\top \\
	&\qquad\qquad\times (\sqrt{\lambda_i}{C}_2 + \lambda_i \sqrt{\lambda_i} {D}_2  B^\top P) < 0,
		\end{aligned}
	\end{equation}	
	for $i=2,\ldots,N$. Taking $P_i = P$ for $i = 2,3,\ldots,N$ and $F  = -c B^{\top}P$ in  \eqref{one_lyap_p} immediately yields (\ref{P}). 
	%
	Next, it follows from \eqref{one_gamma} that also   \eqref{gamma}  holds.
	By Lemma \ref{lem_FFG} then, all systems (\ref{new_systems}) are internally stabilized and $\sum_{i=2}^{N} {J}_i(F,G)<\gamma$. 
	Subsequently, it follows from  Lemma \ref{lem_cost_2} that the protocol (\ref{dyna_protocol_1}) achieves synchronization for the network (\ref{mas_compact_1}) and $J(F,G)<\gamma$.
	%
\end{proof}

\begin{rem}\label{rem_main1}
	\normalfont	
		In Theorem \ref{main1}, in order to select $\gamma$, the following should be done:
\begin{enumerate}[(i)]
	\item 
	First compute a solution $Q >0$ of the Riccati inequality \eqref{one_are_q} and a solution $P >0$ of the Riccati     			inequality \eqref{one_are_p} (or \eqref{are_p2}, depending on the choice of parameter $c$). Note that these solutions exist.
	\item
	Let $S(P,Q) :=\text{tr} (C_1 Q  P Q C_1^\top ) +  \lambda_N \text{\em tr} (  C_2 Q  C_2^\top )$.
	\item 
	Then choose $\gamma >0$ such that $(N -1)S(P,Q) < \gamma$. 	
\end{enumerate}
Obviously, the smaller $S(P,Q)$, the smaller the feasible upper bound $\gamma$. 
It can be shown that, unfortunately, the problem of minimizing $S(P,Q)$ over all $P,Q>0$ that satisfy \eqref{one_are_q} and \eqref{one_are_p} is a nonconvex optimization problem.
However, since smaller $Q$ leads to smaller $\textnormal{tr} \left(C_2 Q  C_2^\top  \right)$ and smaller $P$ and $Q$ leads to smaller $\textnormal{tr} \left(C_1 Q  P Q C_1^\top \right)$  and, consequently, smaller feasible $\gamma$,
	we could therefore try to find $P$ and $Q$ as small as possible.
	In fact, one can find $Q = Q(\epsilon) >0$ to \eqref{one_are_q} by solving 
	\begin{equation}\label{are_qq}
		A Q + Q A^{\top} -   Q C_1^{\top} C_1 Q + E^{\top} E + \epsilon I_n = 0.
	\end{equation}
	with $\epsilon >0$ arbitrary.
	By using a standard argument, it can be shown that $Q(\epsilon)$  decreases as $\epsilon$ decreases, so   	$\epsilon$ should be taken close to 0 in order to get small $Q$.
	%
	Similarly, one can find $P =  P(c,\sigma) >0$ satisfying (\ref{one_are_p}) by solving
	\begin{equation}\label{are}
		A^{\top}P + PA -PB R(c)^{-1}B^{\top}P + \lambda_N C_2^{\top} C_2 +\sigma I_n  = 0
	\end{equation} 
	with  $R(c) = \frac{1}{-c^2 \lambda_N^3 + 2 c \lambda_N}I_n$, where $c$ is chosen as in (\ref{c1}) and $\sigma >0$ arbitrary. Again, it can be shown that $P(c,\sigma)$ decreases with  decreasing  $\sigma$ and $c$.
	Therefore, small $P$ is obtained by choosing $\sigma>0$ close to $0$ and $c = \frac{2}{\lambda_2^2 + \lambda_2 \lambda_N + \lambda_N^2}$.  

	Similarly, if $c$ satisfies (\ref{c2}) corresponding to case \eqref{thm1_case2}, it can be shown that if we choose $\epsilon>0$ and $\sigma>0$ very close to $0$ and $c>0$ very close to $\frac{2}{\lambda_2^2 + \lambda_2 \lambda_N + \lambda_N^2}$,  we find small solutions to the Riccati inequalities \eqref{one_are_q} and  \eqref{are_p2}   in the sense as explained above for case \eqref{thm1_case1}.
\end{rem}


\begin{rem}
	\normalfont
	In Theorem \ref{main1}, exact knowledge of the largest and the smallest nonzero eigenvalue of the Laplacian matrix is used  to compute the local control gains $F$ and $G$.
	We want to remark that our results can be extended to the case that only lower and upper bounds for these eigenvalues are known.
	In the literature, algorithms are given to estimate $\lambda_2$ in a distributed way, yielding lower and upper bounds, see e.g. \cite{ARAGUES20143253}. 
	Also, an upper bound for $\lambda_N$ can be obtained in terms of the maximal node degree of the graph, see e.g. \cite{eigenvaluesL}. 
	Using these lower and upper bounds on the largest and the smallest nonzero eigenvalue  of the Laplacian matrix, results similar to Theorem \ref{main1} can be formulated, see e.g., \cite{Jiao2018} or \cite{han_observer}.
\end{rem}



\section{Simulation Example}\label{sec_simulation}
In this section, we will give a simulation example to illustrate our design method.
Consider a network of $N= 6$ identical agents with dynamics \eqref{mas_decoupled},  
where 
$	A =
	\begin{pmatrix}
		-2 & 2 \\
		-1 &  1
	\end{pmatrix}$,
	$B = 
	\begin{pmatrix}
		0 \\
		1
	\end{pmatrix},$
	$E = 
	\begin{pmatrix}
		0 & 0 \\
		0.5 & 0
	\end{pmatrix},$
	$C_1 = 
	\begin{pmatrix}
		1 & 0
	\end{pmatrix},$
	$D_1 =
	\begin{pmatrix}
		0 & 1
	\end{pmatrix},$
	$C_2 = 
	\begin{pmatrix}
		0 & 1 \\
		0 & 0
	\end{pmatrix},$
	$D_2 = 
	\begin{pmatrix}
		0 \\ 1
	\end{pmatrix}.$
%
The pair $({A}, {B})$ is stabilizable and  the pair $({C}_1, {A})$ is detectable. 
We also have  
${D}_1 {E}^{\top} =
\begin{pmatrix}
	0 & 0
\end{pmatrix}$, 
${D}_2^{\top} {C}_2  =
\begin{pmatrix}
	0 & 0
\end{pmatrix}$ 
and  ${D}_1  {D}_1^{\top} =1$, ${D}_2^\top {D}_2 =1$. 
 We assume that the communication among the six agents is represented by the undirected cycle graph. For this graph, the smallest non-zero and largest eigenvalue of the Laplacian are $\lambda_2 =1$ and $\lambda_6 = 4$. 
%
%
%
Our goal is to design a distributed dynamic output feedback protocol of the form \eqref{dyna_protocol_1} that synchronizes the controlled network and guarantees the associated cost \eqref{cost_FG} to satisfy $J(F,G) < \gamma$.
 Let the desired upper bound for the cost be $\gamma = 17$.  


We adopt the design method given in case \eqref{thm1_case1} of Theorem \ref{main1}. 
First we compute a positive definite solution $P$ to \eqref{one_are_p} by solving the Riccati equation
\begin{equation}\label{simu_are}
	A^{\top} P + P A + (c^2 \lambda_6^3 - 2 c \lambda_6) PB B^{\top} P +\lambda_6 C_2^{\top} C_2 + \sigma I_2 =0
\end{equation}
with $\sigma =0.001$. Moreover, we choose $c = \frac{2}{\lambda_2^2 + \lambda_2 \lambda_6 + \lambda_6^2} = 0.0952$, which is in fact the `best' choice for $c$ as explained in Remark \ref{rem_main1}.
Then, by solving \eqref{simu_are} in Matlab, we compute  a positive definite solution
$	P=
	\begin{pmatrix}
		 0.9048 &  -2.2810 \\
		-2.2810 &   6.9779
	\end{pmatrix}.$ 
Next, by solving the Riccati equation
\begin{equation*}
	A Q + Q A^{\top} - Q C_1^{\top} C_1 Q + E^{\top} E + \epsilon I_2= 0
\end{equation*}
with $\epsilon = 0.001$ in Matlab, we compute a positive definite solution
$	Q = 
	\begin{pmatrix}
		0.5000  &  0.5000 \\
		0.5000  &  0.6250
	\end{pmatrix}.$  
Accordingly, we compute the associated gain matrices
$	F = \begin{pmatrix}
		0.2172  & -0.6646
	\end{pmatrix},
		G = 
	\begin{pmatrix}
		0.5000 &
		0.5000
	\end{pmatrix}^\top.$

As an example, we take the initial states of the agents to be 
$x_{10} =
\begin{pmatrix}
	1 & -2
\end{pmatrix}^\top$, 
$x_{20} =
\begin{pmatrix}
2 & -5
\end{pmatrix}^\top$, 
$x_{30} =
\begin{pmatrix}
3 & 1
\end{pmatrix}^\top$, 
$x_{40} =
\begin{pmatrix}
4 & 2
\end{pmatrix}^\top$, 
$x_{50} =
\begin{pmatrix}
-1 & 2
\end{pmatrix}^\top$ and
$x_{60} =
\begin{pmatrix}
-3 & 1
\end{pmatrix}^\top$,
and  we take the initial states of the  protocol to be zero.
In Figure \ref{sync_x}, we have plotted the controlled state trajectories of the agents. It can be seen that the designed protocol  indeed synchronizes the network.
\begin{figure}[t!]
	\centering
	\includegraphics[height=5cm]{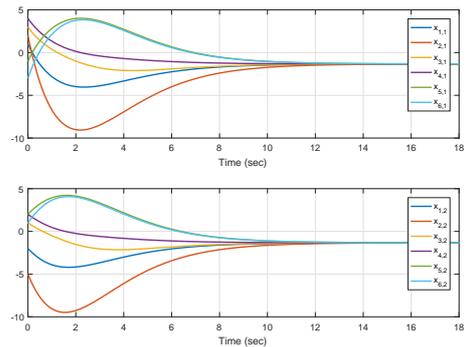}
	\caption{Plots of the state vector $x^1 = (x_{1,1}, x_{2,1}, x_{3,1}, x_{4,1}, x_{5,1}, x_{6,1})^\top$ and $x^2 = (x_{1,2}, x_{2,2}, x_{3,2}, x_{4,2}, x_{5,2}, x_{6,2})^\top$ of the controlled network} \label{sync_x}
\end{figure}
\begin{figure}[t!]
	\centering
	\includegraphics[height=5cm]{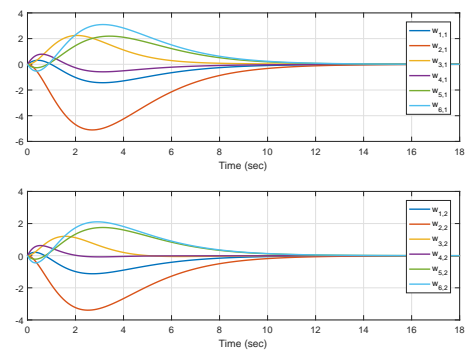}
	\caption{Plots of the state vector $w^1 = (w_{1,1}, w_{2,1}, w_{3,1}, w_{4,1}, w_{5,1}, w_{6,1})^\top$ and $w^2 = (w_{1,2}, w_{2,2}, w_{3,2}, w_{4,2}, w_{5,2}, w_{6,2})^\top$ of the dynamic protocol} \label{sync_w}
\end{figure}
The plots of the protocol states are shown in Figure \ref{sync_w}.
For each $i$, the state $w_i$ of the local controller is an estimate of the weighted sum of the relative states of agent $i$, it is seen that the protocol states converge to zero. 
Moreover, we  compute 
$5  \left(\text{tr} \left(C_1 Q  P Q C_1^\top \right) + \lambda_6 \text{tr} \left(  C_2 Q  C_2^\top  \right)\right)  = 16.6509,$
which is indeed smaller than the desired tolerance $\gamma = 17$.  

\section{Conclusion}\label{sec_conclusion}
In this paper, we have studied the suboptimal distributed $\mathcal{H}_2$ control problem by dynamic output feedback for linear multi-agent systems. The interconnection structure between the agents is given by a connected undirected graph.
Given a linear multi-agent system with identical agent dynamics and an associated global $\mathcal{H}_2$ cost functional, we have provided a design method for computing distributed protocols that guarantee the associated cost to be smaller than a given tolerance while synchronizing the controlled network.
The local gains are given in terms of solutions of two Riccati inequalities, each of dimension equal to that of the agent dynamics. One these Riccati inequalities involves the largest and smallest nonzero eigenvalue of the Laplacian matrix of the network graph. 

%



\bibliographystyle{plain}        
\bibliography{h2_auto}          
\end{document}